\documentclass[a4paper,reqno]{amsart}

\usepackage[colorlinks, linkcolor=red, anchorcolor=blue, citecolor=green]{hyperref}
\usepackage{graphics,epic}
\usepackage{amsmath,amssymb, amsthm,mathrsfs}
\usepackage[all,2cell]{xy}

\newtheorem{theorem}{Theorem}[section]
\newtheorem*{theorem*}{Theorem}

\newtheorem{lemma}[theorem]{Lemma}
\newtheorem{proposition}[theorem]{Proposition}
\newtheorem{corollary}[theorem]{Corollary}

\newtheorem*{conjecture*}{Conjecture}

\newtheorem{thm}[theorem]{Theorem}
\newtheorem{lem}[theorem]{Lemma}

\newtheorem{cor}[theorem]{Corollary}

\newtheorem{df}[theorem]{Definition}

\newcommand{\ie}{{\em i.e.}\ }

\newcommand{\opname}[1]{\operatorname{\mathsf{#1}}}

\renewcommand{\mod}{\opname{mod}\nolimits}

\newcommand{\add}{\opname{add}\nolimits}

\newcommand{\Z}{\mathbb{Z}}

\newcommand{\D}{\mathbb{D}}

\renewcommand{\P}{\mathbb{P}}

\newcommand{\ra}{\rightarrow}

\newcommand{\Hom}{\opname{Hom}}

\newcommand{\Ext}{\opname{Ext}}

\newcommand{\End}{\opname{End}}

\newcommand{\pr}{\opname{pr}}

\newcommand{\mc}{\mathcal{C}}

\newcommand{\mt}{\mathcal{T}}

\newcommand{\homt}{\Hom_{\mathcal{T}}}
\newcommand{\homg}{\Hom_{\Gamma}}

\newcommand{\xra}{\xrightarrow}

\newcommand{\ga}{\Gamma}

\renewcommand{\hat}[1]{\widehat{#1}}

\begin{document}

\title{Relative rigid objects in triangulated categories}\thanks{Partially supported the National Natural Science Foundation of China (Grant No. 11471224)}

\author[Fu]{Changjian Fu}
\address{Changjian Fu\\Department of Mathematics\\SiChuan University\\610064 Chengdu\\P.R.China}
\email{changjianfu@scu.edu.cn}
\date{\today}
\author[Geng]{Shengfei Geng}
\address{Shengfei Geng\\Department of Mathematics\\SiChuan University\\610064 Chengdu\\P.R.China}
\email{genshengfei@scu.edu.cn}
\author[Liu]{Pin Liu}
\address{Pin Liu\\Department of Mathematics\\
   Southwest Jiaotong University\\
  610031 Chengdu \\
  P.R.China}
  \email{pinliu@swjtu.edu.cn}
\maketitle

\begin{abstract}
Let $\mt$ be a Krull-Schmidt, Hom-finite triangulated category with suspension functor $[1]$. Let $R$ be a basic rigid object, $\Gamma$ the endomorphism algebra of $R$, and $\pr(R)\subseteq \mt$ the subcategory of objects finitely presented by $R$. We investigate the relative rigid objects, \ie $R[1]$-rigid objects of $\mt$. Our main results show that the $R[1]$-rigid objects in $\pr(R)$ are in bijection with $\tau$-rigid $\Gamma$-modules, and the maximal $R[1]$-rigid objects with respect to $\opname{pr}(R)$ are in bijection with support $\tau$-tilting $\Gamma$-modules. We also show that various previously known bijections involving support $\tau$-tilting modules are recovered under respective assumptions.
\end{abstract}

\section{Introduction}
This note attempts to unify and generalize certain bijections involving support $\tau$-tilting modules.
Support $\tau$-tilting module is the central notion in the $\tau$-tilting theory introduced by Adachi-Iyama-Reiten~\cite{AIR}, which can be  regarded as a generalization of the classical tilting  module. Since its appearance, support $\tau$-tilting module has been rapidly found to be linked up with various important objects in representation theory, such as torsion class, (co)-$t$-structure, cluster tilting object, silting object and so on, see~\cite{AIR,IJY,CZZ,LX,YZ,YZZ} for instance. Among others, Adachi-Iyama-Reiten~\cite{AIR} proved that for a $2$-Calabi-Yau triangulated category $\mt$ with a basic cluster tilting object $T$, there is a one-to-one correspondence between the set of basic cluster tilting objects of $\mt$ and the set of basic support $\tau$-tilting $\End_\mt(T)$-modules. It is known that there exist  $2$-Calabi-Yau triangulated categories which have no cluster tilting objects but only  maximal rigid objects. Then the correspondence was generalized to such kind of 2-Calabi-Yau triangulated categories by Chang-Zhang-Zhu~\cite{CZZ} and Liu-Xie~\cite{LX}.  The Adachi-Iyama-Reiten's correspondence has been further generalized by Yang-Zhu ~\cite{YZ}. More precisely, let $\mt$ be a Krull-Schmidt, Hom-finite triangulated category with suspension functor $[1]$.
Assume that $\mt$ admits a Serre functor and a cluster tilting object $T$. By introducing the notion of $T[1]$-cluster tilting objects as a generalization of cluster tilting objects, Yang-Zhu~\cite{YZ}  established a one-to-one correspondence between the set of $T[1]$-cluster tilting objects of $\mt$ and the set of support $\tau$-tilting modules over $\End_\mt(T)$. On the other hand, for a Krull-Schmidt, Hom-finite  triangulated category $\mt$ with a silting object $S$, Iyama-J{\o}rgensen-Yang~\cite{IJY} proved that the two-term silting objects of $\mt$ with respect to $S$ which belong to the finitely presented subcategory $\pr(S)$ are in bijection with  support $\tau$-tilting modules over $\End_\mt(S)$. 

   In this note,  we work in the following general setting. 
 Let $\mt$ be a Krull-Schmidt, Hom-finite triangulated category with suspension functor $[1]$ and $R$ a basic rigid object of $\mt$ with endomorphism algebra $\Gamma=\End_\mt(R)$. Denote by $\opname{pr}(R)$ the subcategory of objects finitely presented by $R$. Following~\cite{YZ}, we introduce the $R[1]$-rigid object of $\mt$ and the maximal $R[1]$-rigid object with respect to $\opname{pr}(R)$ ({\it cf.}~Definition~\ref{d:definition}).  We prove that the $R[1]$-rigid objects in $\pr(R)$ are in bijection with $\tau$-rigid $\Gamma$-modules, and the maximal $R[1]$-rigid objects with respect to $\opname{pr}(R)$ are in bijection with support $\tau$-tilting $\Gamma$-modules ({\it cf.}  Theorem~\ref{t:thm1}). When $R$ is a cluster tilting object of $\mt$, we show that the bijection reduces to the bijection between the set of basic $R[1]$-cluster tilting objects of $\mt$ and the set of basic support $\tau$-tilting $\Gamma$-modules obtained by Yang-Zhu~\cite{YZ} (Corollary~\ref{c:YZ}). We remark that, compare to ~\cite{YZ}, we do not need the existence of a Serre functor for $\mt$ ({\it cf.} also~\cite{YZZ}). Since tilting modules are  faithful support $\tau$-tilting modules, we also obtain a characterization of tilting $\Gamma$-modules via the bijection ({\it cf.} Theorem~\ref{t:thm2}).

We apply the aforementioned bijection to the cases of  silting objects, $d$-cluster tilting objects and maximal rigid objects in Section~\ref{S:silting} and Section~\ref{S:d-rigid} respectively. 
 When $R$ is a silting object of a triangulated category $\mt$,  we proved that the maximal $R[1]$-rigid objects with respect to $\opname{pr}(R)$ coincide with the silting objects of $\mt$ in $\opname{pr}(R)$ ({\it cf.}  Theorem~\ref{t:case d=infinite}). As a consequence, Theorem~\ref{t:thm1} recovers the bijection between the set of basic silting objects of $\mt$ in $\opname{pr}(R)$ and the set of basic support $\tau$-tilting $\End(R)$-modules obtained by Iyama-J{\o}rgensen-Yang~\cite{IJY} ({\it cf.}  Corollary~\ref{c:IJY}). If $\mt$ is a $d(\geq 2)$-cluster category and $R$ is a $d$-cluster tilting object of $\mt$, then Theorem~\ref{t:thm1} reduces to the bijection obtained by Liu-Qiu-Xie~\cite{LQX} ({\it cf.}  Corollary~\ref{c:LQX}).
 Assume that $\mt$ is a $2$-Calabi-Yau triangulated category and $R$ is a basic maximal rigid object of $\mt$. We show that Theorem~\ref{t:thm1} implies the bijection between the set of basic maximal rigid object of $\mt$ and the set of basic support $\tau$-tilting $\Gamma$-modules obtained in~\cite{LX,CZZ} ({\it cf.}  Corollary~\ref{c:CZZ-LX}). 

\subsection*{Convention} Let $k$ be an algebraically closed field.  Throughout this paper, $\mt$ will be a Krull-Schmidt, Hom-finite triangulated category over $k$ unless stated otherwise. For an object $M\in \mt$, denote by $|M|$ the number of non-isomorphic indecomposable direct summands of $M$. Denote by $\add M$ the subcategory of $\mt$ consisting of objects which are finite direct sum of direct summands of $M$.

\section{R[1]-rigid objects and $\tau$-rigid modules }~\label{S:basic-notions}

\subsection{Recollection on $\tau$-tilting theory}
We follow~\cite{AIR}. Let $A$ be a finite dimensional algebra over $k$. Denote by $\mod A$ the category of finitely generated right $A$-modules and $\opname{proj} A$ the category of finitely generated right projective $A$-modules. For a module $M\in \mod A$, denote by $|M|$ the number of non-isomorphic indecomposable direct summands of $M$.
 Let $\tau$ be the Auslander-Reiten translation of $\mod A$.

 An $A$-module $M$ is {\it $\tau$-rigid} if $\Hom_A(M,\tau M)=0$.
 A {\it $\tau$-rigid pair} is a pair of $A$-modules $(M,P)$ with $M \in\mod A$ and $P \in \opname{proj} A$, such that $M$ is $\tau$-rigid and $\Hom_A(P,M) = 0$.
A basic $\tau$-rigid pair $(M,P)$ is a  {\it basic support $\tau$-tilting pair} if $|M|+|P| = |A|$. In this case, $M$ is a {\it support $\tau$-tilting} $A$-module and $P$ is uniquely determined by $M$.  It has been proved in~\cite{AIR} that for each $\tau$-rigid pair $(M,P)$, we always have $|M|+|P|\leq |A|$ and 
each $\tau$-rigid pair can be completed into a support $\tau$-tilting pair.

The following criterion for $\tau$-rigid modules has been proved in~\cite{AIR}.
\begin{lem}~\label{l:tau-rigid}
For $M\in \mod A$, denote by $P_1^M\xra{f} P_0^M\ra M\ra 0$ a minimal projective presentation of $M$. Then $M$ is $\tau$-rigid if and only if \[\Hom_{A}(f,M):\Hom_{A}(P_0^M,M)\ra \Hom_{A}(P_1^M,M)\] is surjective.
\end{lem}

\subsection{R[1]-rigid objects}   
Let $\mt$ be a Krull-Schmidt, Hom-finite triangulated category with shift functor $[1]$.  For $X, Y, Z\in \mt$, we denote by $Z(X,Y)$ the subgroup of $\Hom_\mt(X,Y)$ consisting of morphisms which factor through $\add Z$.
 An object $X\in \mt $ is called {\it rigid} if $\homt(X,X[1])=0$. It is {\it maximal rigid} if it is rigid and $\homt(X\oplus Z,X[1]\oplus Z[1])=0$ implies $Z\in \add X$ for any $Z\in \mt$. 
 Let $\mc\subseteq \mt$ be a full subcategory of $\mt$. An object $X\in \mc\subseteq \mt$ is called {\it maximal rigid  with respect to $\mc$} provided that it is rigid and for any object $Z\in \mc$ such that $\homt(X\oplus Z, X[1]\oplus Z[1])=0$, we have $Z\in \add X$. It is clear that a maximal rigid object of $\mt$ is just a maximal rigid object with respect to $\mt$.
 
 Let $R$ be a basic rigid object of $\mt$. An object $X$ is {\it finitely presented} by $R$ if there is a triangle $R_1^X\to R_0^X\to X\to R_1^X[1]$ with $R_0^X,R_1^X\in \add R$. Denote by $\opname{pr}(R)$ the subcategory of $\mt$ consisting of objects which are finitely presented by $R$.  Throughout this section, $R$ will be a basic rigid object of $\mt$.
We introduce the relative rigid objects with respect to $R$ ({\it cf.}~\cite{YZ,CZZ}).
\begin{df}~\label{d:definition}
Let $R\in\mt$ be a basic rigid object.
\begin{enumerate}
\item An object  $X\in\mt$ is called {\it $R[1]$-rigid} if $R[1](X,X[1])=0$.
\item An object $X\in\opname{pr}(R)\subseteq \mt$ is called {\it maximal $R[1]$-rigid with respect to $\opname{pr}(R)$} if  $X$ is $R[1]$-rigid and for any object $Z\in \opname{pr}(R)$ such that $R[1](X\oplus Z, X[1]\oplus Z[1])=0$, then $Z\in\add X$.
\end{enumerate}
\end{df}
By definition, it is clear that rigid objects are $R[1]$-rigid, but the converse is not true in general.
We are interested in $R[1]$-rigid objects of $\mt$ which belong to the subcategory $\opname{pr}(R)$. We have the following observation.
\begin{lem}~\label{l:R[1]-rigid}
 Let  $R_1\xra{f} R_0\xra{g} X\xra{} R_1[1]$ be a triangle in $\mt$ with $R_0,R_1\in\add R$. Then $X$ is $R[1]$-rigid if and only if  $$\Hom_{\mt}(f,X):\Hom_{\mt}(R_0,X)\ra \Hom_{\mt}(R_1,X)$$ is surjective.
\end{lem}
\begin{proof}
  Applying the functor $\homt(-,X)$ to the triangle $X[-1]\xra{h}R_1\xra{f} R_0\xra{g} X$ yields a long exact sequence
\[\homt(R_0,X)\xra{f^*}\homt(R_1,X)\xra{h^*}\homt(X[-1],X)\xra{} \homt(R_0[-1],X),\]
where $f^*=\homt(f,X)$ and $h^*=\homt(h,X)$.

Suppose that $X$ is $R[1]$-rigid, that is $R[1](X,X[1])=0$. It follows that $h^*=0$ and hence $f^*$ is surjective.

Now assume that $f^*$ is surjective. To show $X$ is $R[1]$-rigid, it suffices to prove that $R(X[-1], X)=0$.
  Let $a\in R(X[-1],X)$ and $a=b\circ c$, where $c: X[-1]\ra R$ and $b:R\ra X$. As $R$ is rigid, we know that each morphism from $X[-1]$ to $R$ factors through $h$. 
Hence there is a morphism $c':R_1\ra R$ such that $c=c'\circ h$.  Since $f^*$ is surjective, there is a morphism $b':R_0\ra X$ such that $b\circ c'=b'\circ f$. Then we have $a=b\circ c=b\circ c'\circ h=b'\circ f\circ h=0$ ({\it cf.} the following diagram).
\[
\xymatrix{
X[-1]\ar[r]^{h}\ar[d]^c&R_1\ar[r]^f\ar@{.>}[dl]_{c'}&R_0\ar[r]\ar@{.>}[dl]_{b'}&X\ar[r]&R_1[1]\\
R\ar[r]^b&X&&&
}
\]
\end{proof}

\subsection{From $R[1]$-rigid objects to $\tau$-rigid modules} Recall that $R$ is a basic rigid object of $\mt$.
Denote by $\Gamma:=\End_\mt(R)$ the endomorphism algebra of $R$ and $\mod \Gamma$ the category of finitely generated right $\Gamma$-modules. Let $\tau$ be the Auslander-Reiten translation of $\mod \Gamma$.
 It is known that the functor $\Hom_\mt(R,-): \mt\to \mod \Gamma$ induces an equivalence of categories
\begin{equation}~\label{f:equivalence}
\Hom_\mt(R,-): \opname{pr}(R)/(R[1])\to \mod \Gamma,
\end{equation}
where $\opname{pr}(R)/(R[1])$ is the additive quotient of $\opname{pr}(R)$ by morphisms factorizing through $\add(R[1])$ ({\it cf.} \cite{IY}). Moreover, the restriction of $\Hom_\mt(R,-)$ to the subcategory $\add R$ yields an equivalence between $\add R$ and the category $\opname{proj} \Gamma$ of finitely generated projective $\Gamma$-modules. 
The following result is a direct consequence of the equivalence ~(\ref{f:equivalence}) and the fact that $R$ is rigid.
\begin{lemma}~\label{l:morphism}
For any $R'\in \add R$ and $Z\in \opname{pr}(R)$, we have
\[\homg(\homt(R,R'),\homt(R,Z))\cong \Hom_{\mt}(R',Z).\]
\end{lemma}

Now we are in position to state the main result of this note.
\begin{theorem}~\label{t:thm1}
\begin{itemize}
\item[$(a)$] Let $X$ be an object in $\pr(R)$ satisfying that $\add X\cap\add(R[1])=\{0\}$. Then $X$ is $R[1]$-rigid  if and only if $\Hom_{\mt}(R, X)$ is $\tau$-rigid.

\item[$(b)$] The functor $\Hom_\mt(R, -)$ yields a bijection between the  set of isomorphism classes of basic $R[1]$-rigid objects in $\opname{pr}(R)$ and the set of isomorphism classes of basic $\tau$-rigid pairs of $\Gamma$-modules. 

\item[$(c)$] The functor $\Hom_\mt(R, -)$ induces a bijection between the set of isomorphism classes of basic maximal $R[1]$-rigid objects with respect to $\opname{pr}(R)$ and the set of isomorphism classes of basic support $\tau$-tilting $\Gamma$-modules.
\end{itemize}
\end{theorem}
\begin{proof}
Let $R_1\xra{f}R_0\xra{g} X\to R_1[1]$ be a triangle in $\mt$ with $R_1,R_0\in \add R$ such that $g$ is a minimal right $\add R$-approximation of $X$. As $R$ is rigid and $\add X\cap\add(R[1])=\{0\}$,  applying the functor $\Hom_\mt(R,-)$, we obtain a minimal projective resolution of $\Hom_\mt(R,X)$
\[\homt(R,R_1)\xra{\homt(R,f)}\homt(R,R_0)\ra \homt(R,X)\ra 0.\]

According to Lemma~\ref{l:morphism}, we have the following commutative diagram
\[\xymatrix{\homt(R_0,X)\ar[dd]^{\homt(f,X)}\ar[rr]^-\cong&&\homg(\homt(R,R_0), \homt(R,X))\ar[dd]^{\homg(\homt(R,f),\homt(R,X))}\\
\\
\homt(R_1,X)\ar[rr]^-\cong&&\homg(\homt(R,R_1), \homt(R,X)).
}
\]
Now it follows from Lemma~\ref{l:tau-rigid} and Lemma~\ref{l:R[1]-rigid} that $X$ is $R[1]$-rigid if and only if $\homt(R,X)$ is $\tau$-rigid. This finishes the proof of $(a)$.

Let us consider the statement $(b)$. For each object $X\in \opname{pr}(R)$, $X$ admits a unique decomposition as $X\cong X_0\oplus R_X[1]$, where $R_X\in \add R$ and $X_0$ has no direct summands in $\add R[1]$. We then define 
\[F(X):=(\homt(R,X_0), \homt(R,R_X))\in \mod \Gamma\times \opname{proj} \Gamma.
\]If $X$ is $R[1]$-rigid, according to $(a)$, we deduce that $\homt(R, X_0)$ is a $\tau$-rigid $\Gamma$-module.  And by Lemma~\ref{l:morphism}, we know $\homg(\homt(R,R_X),\homt(R, X_0))=0$. That is, $F$ maps a basic $R[1]$-rigid object to a basic $\tau$-rigid pair of $\Gamma$-modules. We claim that $F$ is the desired bijection.

Since $\homt(R,-):\opname{pr}(R)/(R[1])\to \mod \Gamma$ is an equivalence, we clearly know that $F$ is injective. It remains to show that $F$ is surjective.
For each basic $\tau$-rigid pair $(M,P)$ of $\Gamma$-modules, denote by $\hat{P}\in \add R$ the  object in $\opname{pr}(R)$  corresponding to $P$ and similarly by $\hat{M}\in \opname{pr}(R)$ the object corresponding to $M$, which has no direct summands in $\add R[1]$. By definition, we clearly have $F(\hat{M}\oplus \hat{P}[1])=(M,P)$. It remains to show that $\hat{M}\oplus \hat{P}[1]$ is $R[1]$-rigid, which is a consequence of $(a)$, Lemma~\ref{l:morphism} and the fact that $R$ is rigid. This completes the proof of $(b)$.

For $(c)$, let $X=X_0\oplus R_X[1]$ be a basic maximal $R[1]$-rigid object with respect to $\opname{pr}(R)$, where $R_X\in \add R$ and $X_0$ has no direct summands in $\add R[1]$.   We claim that $F(X)$ is a support $\tau$-tilting pair. Otherwise, at least one of the following two situations happen:
\begin{itemize}
\item[$(i)$] there is an indecomposable object $R_{X^c}\in \add R$ such that $R_{X^c}\not\in \add R_X$ and $(\homt(R, X_0), \homt(R, R_X\oplus R_{X^c}))$ is a basic $\tau$-rigid pair;
\item[$(ii)$] there is an indecomposable object $X_1\in \opname{pr}(R)\backslash \add R[1]$ such that $X_1\not\in \add X_0$ and $(\homt(R, X_0\oplus X_1), \homt(R, R_X))$ is a basic $\tau$-rigid pair.
\end{itemize}
Let us consider the case $(i)$. By definition, we have \[\homg(\homt(R, R_X\oplus R_{X}^c), \homt(R, X_0))=0.\] According to Lemma~\ref{l:morphism}, we clearly have $\homt(R_X\oplus R_{X^c}, X_0)=0$. Now it is straightforward to check that $X\oplus R_{X^c}[1]\in \opname{pr}(R)$ is $R[1]$-rigid. Note that we have $R_{X^c}[1]\not\in \add X$, which contradicts to the assumption that $X$ is a basic maximal $R[1]$-rigid object with respect to $\opname{pr}(R)$. Similarly, one can obtain a contradiction for the case $(ii)$.

Now assume that $(M,P)$ is a basic support $\tau$-tilting pair of $\Gamma$-modules.   According to $(b)$, let $\hat{M}\oplus \hat{P}[1]$ be the basic $R[1]$-rigid object in $\opname{pr}(R)$ corresponding to $(M,P)$. We need to prove that $\hat{M}\oplus \hat{P}[1]$ is maximal with respect to $\opname{pr}(R)$. By definition, we show that if $Z\in \opname{pr}(R)$ is an object such that $R[1](\hat{M}\oplus \hat{P}[1]\oplus Z, \hat{M}[1]\oplus \hat{P}[2]\oplus Z[1])=0$, then $Z\in \add(\hat{M}\oplus \hat{P}[1])$. Without loss of generality, we assume that $Z$ is indecomposable. We separate the remaining proof by considering whether the object $Z$ belongs to $\add R[1]$ or not.

If $Z\not\in \add R[1]$, then $M\oplus \homt(R,Z)$ is a $\tau$-rigid $\Gamma$-module by $(a)$. Moreover, according to $R[1](\hat{M}\oplus \hat{P}[1]\oplus Z, \hat{M}[1]\oplus \hat{P}[2]\oplus Z[1])=0$, we have \[\homg(P, \homt(R, Z))=\homt(\hat{P},  Z)=0.\]  Consequently, $(M\oplus \homt(R,Z), P)$ is a $\tau$-rigid pair. By the assumption that $(M,P)$ is a basic support $\tau$-tilting pair, we conclude that $\homt(R,Z)\in \add M$ and hence $Z\in \add \hat{M}\subseteq \add(\hat{M}\oplus \hat{P}[1])$.

Similarly, for $Z\in \add R[1]$, one can show that $(M, P\oplus \homt(R, Z[-1]))$ is a $\tau$-rigid pair of $\Gamma$-modules. Consequently, we have $Z\in \add \hat{P}[1]\subseteq \add(\hat{M}\oplus \hat{P}[1])$. This completes the proof of $(c)$.
\end{proof}

Since all basic support $\tau$-tilting pairs of $\Gamma$-modules have the same number of non-isomorphic indecomposable direct summands \cite{AIR}. As a byproduct of the proof, we have
\begin{cor}~\label{c:number-direct-summand}
\begin{enumerate}
\item Each $R[1]$-rigid object in $\opname{pr}(R)$ can be completed to a maximal $R[1]$-rigid object with respect to $\opname{pr}(R)$.
\item All basic maximal $R[1]$-rigid objects with respect to $\opname{pr}(R)$ have the same number of non-isomorphic indecomposable direct summands.
\end{enumerate}
\end{cor}

Recall that an object $T\in \mt$ is a {\it cluster tilting object} provided that
\[\add T=\{X\in \mt~|~\homt(T,X[1])=0\}=\{X\in \mt~|~\homt(X,T[1])=0\}.
\]
It is clear that cluster tilting objects are maximal rigid. Let $R$ be a cluster tilting object of $\mt$.  In this case, we have  $\opname{pr}(R)=\mt$ ({\it cf.}~\cite{IY,KZ}).
An object $X\in\mt$ is called {\it $R[1]$-cluster tilting} if $X$ is $R[1]$-rigid and $|X|=|R|$~({\it cf.}~\cite{YZ}).
As a direct consequence of Corollary~\ref{c:number-direct-summand}, we have
\begin{lemma}~\label{l:cluster-tilting}
Let $R$ be a cluster tilting object of $\mt$. Then an object $T\in \mt$ is maximal $R[1]$-rigid with respect to $\mt$ if and only if $T$ is $R[1]$-cluster tilting.
\end{lemma}

Combining Lemma~\ref{l:cluster-tilting} with Theorem~\ref{t:thm1}, we obtain the following result of Yang-Zhu~\cite[Theorem 1.2]{YZ}.
\begin{corollary}~\label{c:YZ}
Let $R$ be a cluster tilting object of $\mt$ with endomorphism algebra $\Gamma=\End_\mt(R)$. There is a bijection between the set of isomorphism classes of basic $R[1]$-cluster tilting objects and the set of isomorphism classes of basic support $\tau$-tilting $\Gamma$-modules.
\end{corollary}

\subsection{A characterization of tilting modules}
Recall that a basic $\Gamma$-module $M$ is a {\it tilting module} provided that
\begin{enumerate}
\item[$\bullet$] $\opname{pd}_\Gamma M\leq 1$;
\item[$\bullet$] $\Ext^1_{\Gamma}(M,M)=0$;
\item[$\bullet$] $|M|=|\Gamma|$.
\end{enumerate}
It has been observed in~\cite{AIR} that tilting $\Gamma$-modules are precisely faithful support $\tau$-tilting $\Gamma$-modules.   As in \cite{LX2014,BBT}, we consider the projective dimension of $\Gamma$-modules and give a characterization of tilting $\Gamma$-modules via $\opname{pr}(R)$.
\begin{theorem}~\label{t:thm2}
For an object $X\in \opname{pr}(R)$ without direct summands in $\add R[1]$, we have
\[\opname{pd}_\Gamma \Hom_\mt(R, X)\leq 1~\text{ if and only if}~ X(R[1], R[1])=0.\] In particular, for a basic object $X\in \opname{pr}(R)$ which has no direct summands in $\add R[1]$,  $\Hom_\mt(R,X)$ is a tilting $\Gamma$-module if and only if $X(R[1], R[1])=0$ and $X$ is maximal $R[1]$-rigid with respect to $\opname{pr}(R)$.
\end{theorem}
\begin{proof}
Since $X\in \opname{pr}(R)$, we have a triangle $R_1\xra{f}R_0\xra{g} X\xra{h} R_1[1]$ in $\mt$ such that $R_0,R_1\in \add R$ and $g$ is a minimal right $\add R$-approximation of $X$. Applying the functor $\homt(R,-)$, we obtain
 a long exact sequence
\[\homt(R,X[-1])\xra{\homt(R,h[-1])}\homt(R, R_1)\xra{\homt(R,f)}\homt(R,R_0)\xra{}\homt(R,X)\to 0.
\]

Assume that $X(R[1],R[1])=0$. It follows that $\homt(R,h[-1])=0$ and $\homt(R,f)$ is injective. That is, $\opname{pd}_\Gamma \homt(R,X)\leq 1$.

Suppose that $\opname{pd}_\Gamma \homt(R,X)\leq 1$. Then $\homt(R,f)$ is injective and $\homt(R,h[-1])=0$. It suffices to prove that $X[-1](R,R)=0$.
Let $s: R\to X[-1]$ be a morphism from $R$ to $X[-1]$ and $t: X[-1]\to R$ a morphism from $X[-1]$ to $R$. We need to show that $t\circ s=0$.
Since $R$ is rigid, we know that the morphism $t:X[-1]\to R$ factors through the morphism $h[-1]$. In particular, there is a morphism $t':R_1\to R$ such that $t=t'\circ h[-1]$.
On the other hand, by $\homt(R,h[-1])=0$, we deduce that $h[-1]\circ s=0$. Consequently, $t\circ s=t'\circ h[-1]\circ s=0$.

Now we assume that $X$ is a maximal $R[1]$-rigid object with respect to $\opname{pr}(R)$ such that $X(R[1],R[1])=0$. By Theorem~\ref{t:thm1}, $\homt(R,X)$ is a support $\tau$-tilting $\Gamma$-module. Since $X$ does not admit an indecomposable direct summand in $\add R[1]$, we have $|\homt(R,X)|=|X|=|R|=|\Gamma|$ by Corollary~\ref{c:number-direct-summand}. The condition $X(R[1],R[1])=0$ implies that $\opname{pd}_\Gamma \homt(R,X)\leq 1$. Putting all of these together, we conclude that $\homt(R,X)$ is a tilting $\Gamma$-module.

Conversely, let us assume that $\homt(R,X)$ is a tilting $\Gamma$-module. Since tilting modules are support $\tau$-tilting modules, we know that $X$ is a maximal $R[1]$-rigid object with respect to $\opname{pr}(R)$ by Theorem~\ref{t:thm1} (c). By definition of tilting modules, we have $\opname{pd}_\Gamma \homt(R,X)\leq 1$. Consequently, $X(R[1],R[1])=0$ and we are done.
\end{proof}

\section{ $R[1]$-rigid objects and presilting objects}~\label{S:silting}
\subsection{(Pre)silting objects}
Recall that $\mt$ is a Krull-Schmidt, Hom-finite triangulated category with shift functor [1].  Following~\cite{AI},
 for $X,Y\in \mt$ and $m\in \Z$,  we write the vanishing condition $\homt(X,Y[i])=0$ for $i>m$ by $\homt(X,Y[>m])=0$. 
An object $X\in\mt$ is called {\it presilting} if $\homt(X,X[>0])=0$; $X$ is called {\it silting} if $X$ is presilting and the thick subcategory of $\mt$ containing $X$  is $\mt$; $X$ is called {\it partial silting} if $X$ is a direct summand of some silting objects. 

It is clear that (pre)silting objects are rigid. The following result has been proved in~\cite{AI}.
\begin{lem}\label{l:number-direct-summands-silting}
All silting objects in $\mt$ have the same number of non-isomorphic indecomposable summands.
\end{lem}

In general, it is not known that whether a presilting object is partial silting.  The following is proved in~\cite{A}.
\begin{lem}\label{l:2-term presilting is partial silting}
Let $R$ be a silting object and $X$ a presilting object of $\mt$. If  $X\in \opname{pr}(R)$, then there is a presilting object $Y\in \opname{pr}(R)$ such that $X\oplus Y$ is a silting object of $\mt$. 
\end{lem}

\subsection{From $R[1]$-rigid objects to (pre)silting objects }

\begin{lem} \label{l:presilting case}
Let $R$ be a presilting object and  $X\in \opname{pr}(R)$.  Then $\homt(X,X[>1])=0$.
\end{lem}

\begin{proof}
As $X\in \opname{pr}(R)$, we have the following triangle 
\begin{equation}\label{tr:presilting}
R_1\xra{f}R_0\xra{g}X\xra{h}R_1[1],
\end{equation}
where $R_0,R_1\in \add R$.
Applying the functor $\homt(R,-)$ to the triangle yields a 
long exact sequence
\[\cdots \to \homt(R, R_0[i])\to\homt(R,X[i])\to  \homt(R, R_1[i+1])\cdots.
\]
Then the assumption that $R$ is presilting implies that 
\begin{equation}\label{eq:presilting}
\homt(R, X[>0])=0.
\end{equation}
On the other hand, applying the functor $\homt(-,X[i])$ to the triangle \eqref{tr:presilting}, we obtain a long exact sequence
\[\cdots\to\homt(R_1[1],X[i])\ra \homt(X,X[i])\ra \homt(R_0,X[i])\to \cdots.
\]
Then \eqref{eq:presilting} implies that $\homt(X,X[>1])=0$.
 \end{proof}

\begin{thm}\label{t:case d=infinite}
Let $X$ be an object in $\opname{pr}(R)$.  
\begin{enumerate}
\item If $R$ is a presilting  object, then the followings are equivalent.
\begin{itemize}
\item[(a)] $X$ is an  $R[1]$-rigid object;
\item[(b)] $X$ is a rigid object;
\item[(c)] $X$ is a presilting object. 
\end{itemize} 
\item  If $R$ is a silting object,  then  $X$ is a maximal $R[1]$-rigid object with respect to $\opname{pr}(R)$ if and only if $X$ is a silting object. 
 \end{enumerate}
\end{thm}
\begin{proof}
For $(1)$, according to Lemma~\ref{l:presilting case}, it suffices to prove that each $R[1]$-rigid object is rigid.

Let us assume that $X$ is an $R[1]$-rigid object in $\opname{pr}(R)$. Then there exists a triangle 
\begin{equation}\label{tr:rigid}
R_1\to R_0\to X\xra{h} R_1[1]
\end{equation} with $R_0,R_1\in \add R$. By applying the functor $\homt(R,-)$ to the triangle \eqref{tr:rigid}, we obtain an exact sequence
\[\cdots\to \homt(R, R_0[1])\to \homt(R, X[1])\to \homt(R, R_1[2])\to \cdots.
\]
Since $R$ is a presilting object and $R_0,R_1\in \add R$, we have \[\homt(R, R_0[1])=0=\homt(R,R_1[2]).\] Consequently, $\homt(R,X[1])=0$.  Now applying the functor $\homt(-,X[1])$ to \eqref{tr:rigid}, we obtain an exact sequnece
\[\homt(R_1[1],X[1])\ra \homt(X,X[1])\ra 0.
\]
In other words, each morphism from $X$ to $X[1]$ factors through the morphism $h: X\to R_1[1]$. However, we have $R[1](X,X[1])=0$, which implies that $\homt(X,X[1])=0$ and hence $X$ is rigid. This completes the proof of $(1)$.

Now suppose that $R$ is a silting object. If $X$ is a silting object, then $X$ is an $R[1]$-rigid object by $(1)$. By Lemma \ref{l:number-direct-summands-silting}, we have $|X|=|R|$.  Hence, $X$ is a maximal $R[1]$-rigid object with respect to $\opname{pr}(R)$ by Corollary \ref{c:number-direct-summand}.

On the other hand, if $X$ is a   maximal $R[1]$-rigid object, then $X$ is a  presilting object by $(1)$. Since $X\in \opname{pr}(R)$,  $X$ is a partial silting object by Lemma \ref{l:2-term presilting is partial silting}. According to  Corollary \ref{c:number-direct-summand}, we know that  $|X|=|R|$. Therefore,  $X$ must be a silting object by Lemma \ref{l:number-direct-summands-silting}. 
\end{proof}

Combining Theorem~\ref{t:thm1} with Theorem~\ref{t:case d=infinite}, we obtain the following bijection, which is due to Iyama-J{\o}rgensen-Yang ({\it cf.}~\cite[Theorem~0.2]{IJY}).
\begin{corollary}~\label{c:IJY}
Let $R$ be a basic silting object of $\mt$ with endomorphism algebra $\Gamma=\End_\mt(R)$. There is a bijection between the set of presilting objects which belong to $\opname{pr}(R)$ and the set of $\tau$-rigid pair of $\Gamma$-modules, which induces a one-to-one correspondence between the set of  silting objects in $\opname{pr}(R)$ and the set of support $\tau$-tilting $\Gamma$-modules.
\end{corollary}

\section{R[1]-rigid objects and $d$-rigid objects in $(d+1)$-Calabi-Yau category}~\label{S:d-rigid}

 Let $d$ be a positive integer.  Throughout this section, we assume that $\mt$ is $(d+1)$-Calabi-Yau, \ie we are given bifunctorial isomorphisms \[\homt(X,Y)\cong\D\homt(Y,X[d+1]) ~\text{for}~ X,Y\in \mt,\]
where $\D=\Hom_k(-,k)$ is the usual duality over $k$.

\subsection{From $R[1]$-rigid objects to $d$-rigid objects}
An object $T\in\mt$ is called {\it $d$-rigid} if $\homt(T,T[i])=0$ for $i=1, 2, \cdots, d$. 
It is  {\it maximal $d$-rigid} if $T$ is $d$-rigid and for $i=1,\cdots,d, ~\homt(T\oplus X,(T\oplus X)[i])=0$  implies that $X\in\add T$.
An object $T\in \mt$ is a {\it $d$-cluster-tilting object}  if $T$ is $d$-rigid and for $i=1,\cdots,d$, $\homt(T, X[i])=0$ implies that $X\in \add T$.

 \begin{lem}~\label{l:rigid-d-rigid}
Let $R$ be a $d$-rigid object of $\mt$ and $X\in \opname{pr}(R)$.  Then $X$ is rigid if and only if $X$ is $d$-rigid.
\end{lem}

\begin{proof}
It is obvious that a $d$-rigid object is rigid.

Now suppose that $X$ is rigid. As $X\in \opname{pr}(R)$, we have a triangle 
\begin{equation}\label{eq: rigid-d-rigid}
R_1\xra{f}R_0\xra{g}X\xra{h}R_1[1]
\end{equation}
with $R_0,R_1\in \add R$. Note that we have $\homt(R,R[i])=0$ for $i=1,2,\cdots, d$.  Applying the functor $\homt(R, -)$ to the triangle \eqref{eq: rigid-d-rigid}, we obtain  a
long exact sequence
\[\cdots\to\homt(R, R_0[i])\to \homt(R, X[i])\to \homt(R,R_1[i+1])\to \cdots.
\]
Consequently, 
\begin{equation}\label{eq: relation}
\homt(R, X[i])=0, \quad i=1,\cdots, d-1.
\end{equation}
On the other hand, applying the functor $\homt(-, X[i])$ to the triangle \eqref{eq: rigid-d-rigid} yields a long exact sequence
\[\cdots\to \homt(R_1[1],X[i])\ra \homt(X,X[i])\ra \homt(R_0,X[i])\to \cdots.
\]
Then \eqref{eq: relation} implies that $\homt(X,X[i])=0$ for $i=2,\cdots, d-1$.

Recall that $X$ is rigid and $\mt$ is $(d+1)$-Calabi-Yau, we have 
\[\homt(X,X[d])\cong \D\homt(X,X[1])=0.
\]
Hence $X$ is a $d$-rigid object of $\mt$.
 \end{proof}
 
\begin{thm}\label{t:case d>1}
Let $R\in\mt$ be a $d$-rigid object and $X\in \opname{pr}(R)$. Then the followings are equivalent.
\begin{enumerate}
\item $X$ is an $R[1]$-rigid object.
\item $X$ is a rigid object.
\item $X$ is a $d$-rigid object. 
\end{enumerate}
\end{thm}
\begin{proof}
 According to Lemma~\ref{l:rigid-d-rigid}, it suffices to prove that each $R[1]$-rigid object is rigid. 

 Let us first consider the case that $d\geq 2$. Applying the functor $\homt(R,-)$ to the triangle \eqref{eq: rigid-d-rigid} yields a long exact sequence
 \[\cdots\to \homt(R, R_0[1])\to \homt(R, X[1])\to \homt(R, R_1[2])\to \cdots.
 \]
 As $R$ is $d$-rigid and $R_0,R_1\in \add R$, we conclude that $\homt(R,X[1])=0$.
 Applying the functor $\homt(-,X[1])$ to the triangle \eqref{eq: rigid-d-rigid}, we obtain an exact sequence
 \[\homt(R_1[1],X[1])\xrightarrow{\homt(h, X[1])} \homt(X,X[1])\ra 0.
 \]
In particular, each morphism from $X$ to $X[1]$ factors through the morphism $h:X\to R_1[1]$. Hence the assumption that $X$ is an $R[1]$-rigid object implies that $X$ is rigid.

Now suppose that $d=1$. In this case, $\mt$ is a $2$-Calabi-Yau triangulated category.  Applying the functor $\homt(-,X[1])$ to \eqref{eq: rigid-d-rigid} yields a long exact sequence
\[\cdots\to \homt(R_1[1],X[1])\xra{\homt(h,X[1])}\homt(X,X[1])\xra{\homt(g, X[1])}\homt(R_0,X[1])\to \cdots.
\]
Then the assumption that $X$ is $R[1]$-rigid implies that $\homt(g,X[1])$ is injective. Consequently, the morphism 
\[\D\homt(g,X[1]): \D\homt(R_0,X[1])\to \D\homt(X,X[1])
\]
is surjective.  Thanks to the $2$-Calabi-Yau property, we have the following commutative diagram
\[\xymatrix{\D\homt(R_0,X[1])\ar[d]^{\cong}\ar[rr]^{\D\homt(g,X[1])}&&\D\homt(X,X[1])\ar[d]^{\cong}\\
\homt(X,R_0[1])\ar[rr]^{\homt(X,g[1])}&&\homt(X,X[1]).
}
\]
In particular, $\homt(X,g[1]):\homt(X,R_0[1])\to \homt(X,X[1])$ is surjective.  Again, $R[1](X,X[1])=0$ implies that $\homt(X,g[1])=0$ and then $\homt(X,X[1])=0$. 
\end{proof}

\subsection{$d$-cluster-tilting objects in $d$-cluster categories}
This subsection concentrates on $d$-cluster categories,  a special class of $(d+1)$-Calabi-Yau triangulated categories. We refer to \cite{K, T} for definitions  and ~\cite{ZZ09,W} for basic properties of $d$-cluster categories. Among others, the following result proved in \cite{ZZ09,W} is useful.

\begin{lem}~\label{l:d cluster tilting objects are equivalent to maximal rigid object}
Let $\mt$ be a $d$-cluster category. Then an object $T$ is a  $d$-cluster tilting object if and only if $T$ is a maximal $d$-rigid objects. Moreover, all $d$-cluster tilting objects have the same number of non-isomorphic indecomposable direct summands.
\end{lem}
Then for relative rigid objects, we have the following.
\begin{proposition}\label{p:case d>1-2}
Let $\mt$ be a $d$-cluster category and $R$ be a $d$-cluster tilting object in $\mt$. Assume $X\in \opname{pr}(R)$, then  $X$ is a maximal $R[1]$-rigid object with respect to $\opname{pr}(R)$ if and only if $X$ is a $d$-cluster tilting object.
\end{proposition}
\begin{proof}
Let $X$ be a maximal $R[1]$-rigid object with respect to $\opname{pr}(R)$.  By Corollary \ref{c:number-direct-summand}, we have $|X|=|R|$. According to Theorem \ref{t:case d>1},   $X$ is $d$-rigid and hence a maximal $d$-rigid object in $\mt$. 
It follows from Lemma \ref{l:d cluster tilting objects are equivalent to maximal rigid object} that
$X$ is a $d$-cluster tilting object in $\mt$.

Conversely, assume that $X$ is a $d$-cluster tilting object. According to Lemma~\ref{l:d cluster tilting objects are equivalent to maximal rigid object}, $X$ is a maximal $d$-rigid object with $|X|=|R|$ and hence a maximal $R[1]$-rigid object respect to $\opname{pr}(R)$ by Theorem~\ref{t:case d>1}.
\end{proof}

Combining Theorem~\ref{t:thm1}, Thereom~\ref{t:case d>1} with Proposition~\ref{p:case d>1-2}, we obtain the following main result of~\cite{LQX}.
\begin{corollary} ~\label{c:LQX}
Assume that $d\ge 2$. Let $\mt$ be a $d$-cluster category with a $d$-cluster-tilting object $R$. Denote by $\Gamma=\End_\mt(R)$ the endomorphism algebra of $R$.
The functor $\homt(R,-)$ yields a bijection between the set of isomorphism classes of $d$-rigid objects of $\mt$ which belong to $\opname{pr}(R)$ and the set of isomorphism classes of $\tau$-rigid $\ga$-modules. The bijection induces a bijection between the set of isomorphism classes of $d$-cluster tilting objects of $\mt$ which belong to $\opname{pr}(R)$ and  the set of isomorphism classes of support $\tau$-tilting $\ga$-modules.
\end{corollary}
\subsection{Maximal rigid objects in $2$-Calabi-Yau categories}
In this subsection, we assume that $\mt$ is a $2$-Calabi-Yau category and $R$ a basic maximal rigid object of $\mt$.
It has been proved in~\cite{BIRS, ZZ11} that each rigid object of $\mt$ belongs to $\opname{pr}(R)$.

The following proposition is a direct consequence of Theorem~\ref{t:case d>1} and Corollary~\ref{c:number-direct-summand}.
\begin{proposition}\label{p:case d=1}
Let $\mt$ be a $2$-Calabi-Yau triangulated category and $R$ a basic maximal rigid object of $\mt$.  Let $X$ be an object of $\mt$,  then $X$ is a maximal $R[1]$-rigid object with respect to $\opname{pr}(R)$ if and only if $X$ is a maximal rigid object of $\mt$.
\end{proposition}

Combining Theorem~\ref{t:thm1} with Proposition~\ref{p:case d=1}, we obtain the main result of ~\cite{CZZ,LX}.
\begin{corollary}~\label{c:CZZ-LX}
Let $\mt$ be a $2$-Calabi-Yau category with a basic maximal rigid object $R$. Denote by $\Gamma=\End_\mt(R)$ the endomorphism algebra of $R$.
Then there is a bijection between the set of isomorphism classes of rigid objects of $\mt$ and the set of isomorphism classes of $\tau$-rigid $\ga$-modules, which induces a bijection between the set of isomorphism classes of maximal rigid objects of $\mt$ and the set of isomorphism classes of support $\tau$-tilting $\ga$-modules.
\end{corollary}

\end{document}